\documentclass[11pt]{amsart}
\usepackage{amssymb}
\usepackage{pgfplots}
\usepackage{csvsimple}
\usepackage{siunitx}
\usepackage{hyperref}
\usepackage{graphicx,color}
\usepackage{subcaption}
\usepackage[normalem]{ulem}

\def\R{\mathbb{R}}

\def\N{\mathbb{N}}

\def\cA{\mathcal{A}}

\def\cI{\mathcal{I}}

\def\cM{\mathcal{M}}
\def\cN{\mathcal{N}}
\def\cO{\mathcal{O}}

\def\cS{\mathcal{S}}
\def\cT{\mathcal{T}}

\def\a{\alpha}
\def\b{\beta}
\def\g{\gamma}
\def\d{\delta}

\def\veps{\varepsilon}
\def\vrho{\varrho}
\def\vphi{\varphi}
\def\O{\Omega}

\newcommand{\dv}[1]{\,{\mathrm d}#1}
\newcommand{\dual}[2]{\langle#1,#2\rangle}

\newcommand{\wcheck}[1]{#1\hspace{-.8ex}\mbox{\huge {\lower.45ex \hbox{$\textstyle \check{}$}}} \hspace{.5ex}}

\DeclareMathOperator{\id}{id}
\let\oldmarginpar\marginpar
\renewcommand\marginpar[1]{
  \oldmarginpar[\raggedleft\footnotesize #1]
  {\raggedright\footnotesize #1}}
\newtheorem{definition}{Definition}

\newtheorem{proposition}[definition]{Proposition}

\newtheorem{remark}[definition]{Remark}
\newtheorem{remarks}[definition]{Remarks}
\newtheorem{example}[definition]{Example}

\newtheorem{algorithm}[definition]{Algorithm}

\numberwithin{definition}{section}

\def\hP{\widehat{P}}
\def\hI{\widehat{I}}
\def\tpi{\widetilde{\pi}}

\def\tphi{\widetilde{\phi}}
\def\tpsi{\widetilde{\psi}}

\begin{document}
\title[Discrete optimal transport]{Error bounds for discretized 
optimal transport and its reliable efficient numerical solution}
\author[S. Bartels]{S\"oren~Bartels}
\address{Abteilung f\"ur Angewandte Mathematik,  
Albert-Ludwigs-Universit\"at Freiburg, Hermann-Herder Str. 10, 
79104 Freiburg i.Br., Germany}
\email{bartels@mathematik.uni-freiburg.de}
\author[S. Hertzog]{Stephan Hertzog}
\address{Abteilung f\"ur Angewandte Mathematik,  
Albert-Ludwigs-Universit\"at Freiburg, Hermann-Herder Str. 10, 
79104 Freiburg i.Br., Germany}
\email{stephanhertzog@gmail.com}
\date{\today}
\subjclass{65K10, 49M25, 90C08}
\begin{abstract}
The discretization of optimal transport problems often leads to large
linear programs with sparse solutions. We derive error estimates for
the approximation of the problem using convex combinations of Dirac
measures and devise an active-set strategy that uses the optimality
conditions to predict the support of a solution within a multilevel
strategy. Numerical experiments confirm the theoretically predicted
convergence rates and a linear growth of effective problem sizes 
with respect to the variables used to discretize given data. 
\end{abstract}
\keywords{Optimal transport, sparsity, optimality conditions, 
error bounds, iterative solution}
\maketitle


\section{Introduction}
The goal in {\em optimal transportation} is to transport a measure $\mu$ 
into a measure $\nu$ with minimal total effort with respect to a given cost 
function $c$. This optimization problem can be formulated as an 
infinite-dimensional linear program. One way to find optimal solutions 
is to approximate the transport problem by (finite-dimensional) standard
 linear programs. This can be done by approximating the measures $\mu$ 
and $\nu$ by convex combinations of Dirac measures and we prove that this
leads to accurate approximations of optimal costs. The size of these 
linear programs grows quadratically in the size of the supports of these 
approximations, i.e., if $M$ and $N$ are the number of atoms on which the 
approximations are supported, then the size of the linear programs is $MN$. 
Thus, they can only be solved directly on coarse grids, i.e., for small 
$M$ and $N$. It is another goal of this article to devise an iterative 
strategy that automatically identifies the support of a solution using 
auxiliary problems of comparable sizes. For other approaches to the numerical
solution of optimal transport problems we refer the reader 
to~\cite{RusUck00,BenBre00,BeFrOb14,BenCar15,BarSch17-pre}; for details on the 
mathematical formulation and its analytical features 
we refer the reader to~\cite{Eva99,Vil03,Vil08}.

Our error estimate follows from identifying convex combinations of
Dirac measures supported in the nodes of a given triangulation as 
approximations of probability measures via the adjoint of the standard
nodal interpolation operator defined on continuous functions. Thereby,
it is possible to quantify the approximation quality of a discretized
probability measure in the operator norm related to a class of 
continuously differentiable functions. 

Using the fact that if~$c$ is strictly convex and~$\mu$ has a density, 
the support of optimal solutions is contained in a lower dimensional 
set, we expect 
that the linear programs have a sparse solution, i.e., the number of 
nonzero entries in the solution matrix is comparable to $M+N$. Related 
approaches have previously been discussed in the literature, 
cf.~\cite{ObeRua15,Sch16}. In this article we aim at investigating a 
general strategy that avoids assumptions on an initial guess or 
a coarse solution and 
particular features of the cost function and thus leads to 
an efficient solution procedure that is fully reliable. 

The optimality conditions for standard linear programs characterize the 
optimal support using the Lagrange multipliers $\phi$ and $\psi$ which 
occur as solutions of the dual problem. Given approximations of those 
multipliers, we may restrict the full linear program to the small set 
of atoms where those approximations satisfy the characterizing equations 
of the optimal support up to some tolerance, with the expectation that 
the optimal support is contained in this set. If the solution of the 
corresponding reduced linear program satisfies the optimality conditions 
of the full problem, a global solution is found. Otherwise, the tolerance 
is increased to enlarge the active set of the reduced problem, and the 
procedure is repeated. Good approximations of the Lagrange multipliers 
result from employing a multilevel scheme and in each step prolongating 
the dual solutions computed on a coarser grid to the next finer grid. 

Our numerical experiments reveal that this iterative strategy leads to 
linear programs whose dimensions are comparable to $M+N$. 
The optimality conditions have to be checked on the full product grid which 
requires $\cO(MN)$ arithmetic operations. These are however fully independent
and can be realized in parallel. The related algorithm of \cite{ObeRua15} 
avoids this test and simply adds atoms in a neighborhood of a coarse grid 
solution. This is an efficient strategy if a good coarse grid solution
is available. 

Another alternative is the method presented in \cite{Sch16} where the 
concept of shielding neighbourhoods is introduced. Solutions which are 
optimal in a {\em shielding neighbourhood} are analytically shown to be 
globally optimal. Strategies to construct those sets are presented for 
several cost functions. However, each cost function requires a particular
 strategy to find the neighbourhoods, depending on its geometric structure. 
Critical for the efficiency of the algorithm is the sparsity of shielding 
neighbourhoods for which theoretical bounds and intuitive arguments are 
given, confirmed by numerical experiments.

The efficiency of our numerical scheme can be greatly increased if it is 
combined with the methods from~\cite{ObeRua15} or~\cite{Sch16}. In this 
case the activation of atoms is only done within the described neighbourhoods
of the support of a current approximation. This is expected to be reliable 
once asymptotic convergence behaviour is observed. 

The outline of this article is as follows. The general optimal transport
problem, its discretization, optimality conditions, and sparsity properties
are discussed in Section~\ref{sec_ot}. A rigorous error analysis for 
optimal costs based
on the approximation of marginal measures via duality is carried out 
in Section~\ref{sec_erroranalysis}. Section~\ref{sec_strategy} devises the 
multilevel active set stategy for efficiently solving the linear programs 
arising from the discretization. The efficiency of the algorithm and 
the optimality of the error estimates are illustrated via
numerical experiments in Section~\ref{sec_exp}.


\section{Discretized Optimal Transport} \label{sec_ot}
We describe in this section the general mathematical framework for 
optimal transport problems, their discretization, optimality conditions,
and sparsity properties of optimal transport plans. 

\subsection{General formulation} 
The general form of an optimal transport problem seeks a 
probability measure $\pi \in \cM(X\times Y)$ called a {\em transport plan}
on probability spaces $X$ and $Y$ such that its projections onto 
$X$ and $Y$ coincide with given probability measures $\mu\in\cM(X)$ and $\nu\in\cM(Y)$, 
respectively, called {\em marginals}, and such that it is optimal in the 
set of all such measures for a given continuous {\em cost function} 
$c:X\times Y\to \R$. The minimization problem thus reads:
\[
(\hP) \quad \left\{ 
\begin{array}{l}
\text{Minimize } \hI[\pi] = \iint_{X\times Y} c(x,y) \dv{\pi(x,y)} \\[1.5mm]
\text{subject to } \pi \in \cM(X\times Y), \, \pi \ge 0, \, P_X \pi = \mu, \, P_Y \pi = \nu
\end{array} 
\right. 
\]
Here, $P_X \pi$ and $P_Y \pi$ are defined via $P_X \pi(A) = \pi(A\times Y)$
and $P_Y \pi (B) = \pi(X\times B)$ for measurable sets $A\subset X$ and $B\subset Y$, 
respectively. 
This formulation may be regarded as a relaxation of the problem of 
determining a {\em transport map} $T:X\to Y$ which minimizes a cost 
functional in the set of bijections between $X$ and $Y$ 
subject to the constraint that the measure $\mu$ is pushed
forward by $T$ into the measure $\nu$: 
\[
(P) \quad \left\{ 
\begin{array}{l}
\text{Minimize } I[T] = \int_X c(x,T(x)) \dv{\mu(x)} \\[1.5mm]
\text{subject to $T$ bijective and } T_\# \mu = \nu
\end{array} 
\right. 
\]
In the case that $\mu$ and $\nu$
have densities $f\in L^1(X)$ and $g\in L^1(Y)$ the relation 
$T_\# \mu = \nu$ is equivalent to the identity  
\[
g\circ T \det DT = f,
\]
which is a Monge--Amp\`ere equation if $T= \nabla \Phi$ for a convex potential~$\Phi$.
Since the formulation $(P)$ does not provide sufficient
control on variations of transport maps to pass to limits in
the latter equation, it is difficult to establish the existence of solutions
directly. In fact, optimal transport maps may not exist, e.g., when
a single Dirac mass splits into a convex combination of several 
Dirac masses. The linear 
program $(\hP)$ extends the formulation $(P)$ via graph
measures $\pi = (\id \times T)\# \mu$ 
and admits solutions. In the case of a strictly convex cost function $c$ 
it can be shown that optimal transport plans correspond to optimal transport 
maps, i.e., optimal 
plans are supported on graphs of transport maps, provided that $\mu$ has a density. 
In this sense $(\hP)$ is a relaxation of $(P)$;
we refer the reader to~\cite{Eva99,Vil03,Vil08} for details. 

\subsection{Discretization} 
In the case where the marginals are given by convex 
combinations of Dirac measures supported in atoms 
$(x_i)_{i=1,\dots,M} \subset X$ and 
$(y_j)_{j=1,\dots,N} \subset Y$, respectively, i.e.,  
\[
\mu_h = \sum_{i=1}^M \mu_h^i \d_{x_i}, \quad \nu_h = \sum_{j=1}^N \nu_h^j \d_{y_j}, 
\]
we have that admissible transport plans $\pi$ are supported in the set of
pairs of atoms $(x_i,y_j)$.
Indeed, if $A\times B \subset X\times Y$ with $(x_i,y_j)\not \in A\times B$, i.e.,
$x_i \not \in A$ for all $i\in \{1,2,\dots,M\}$ or $y_j\not \in B$ for all 
$j\in \{1,2,\dots,N\}$, then one of the inequalities 
\[\begin{split}
\pi(A\times B) &\le \pi(A\times Y) = \mu_h(A) = 0, \\
\pi(A\times B) &\le \pi(X\times B) = \nu_h(B) = 0,
\end{split}\]
holds, and we deduce $\pi(A\times B)= 0$. 
By approximating measures $\mu$ and $\nu$ by convex combinations of Dirac 
measures $\mu_h$ and $\nu_h$, we therefore directly obtain a standard linear 
program that determines the unknown matrix $\pi_h \in \R^{M\times N}$: 
\[
(\hP_h) \quad \left\{ 
\begin{array}{l}
\text{Minimize } \hI_h[\pi_h] 
= \sum_{i=1}^M \sum_{j=1}^N c(x_i,y_j)  \pi_h^{ij} \\[1.5mm]
\text{subject to } \pi_h \ge 0, \ 
\sum_{j=1}^N \pi_h^{ij} = \mu_h^i, \ \sum_{i=1}^M \pi_h^{ij} = \nu_h^j
\end{array} 
\right.
\]
The rigorous construction of approximating measures $\mu_h$ and $\nu_h$ 
via duality arguments will be described below in Section~\ref{sec_erroranalysis}.
Weak convergence of discrete transport plans to optimal transport 
plans can be established via abstract theories, cf.~\cite{Vil08,ObeRua15}
for details.

\subsection{Optimality conditions}
Precise information about the support of an optimal discrete transport plan $\pi_h$ 
are provided by the Lagrange multipliers corresponding to the marginal constraints. 
Including these in an augmented Lagrange functional $\widehat{L}_h$ leads to
\[\begin{split}
\widehat{L}_h[\pi_h;\phi_h,\psi_h] 
&= \hI_h[\pi_h] 
+ \sum_{i=1}^M \phi_h^i \Big(\mu_h^i - \sum_{j=1}^N \pi_h^{ij}\Big)
+ \sum_{j=1}^N \psi_h^j \Big(\nu_h^j - \sum_{i=1}^M \pi_h^{ij}\Big) \\
&= \sum_{i=1}^M \sum_{j=1}^N \pi_h^{ij} \Big( c(x_i,y_j) - \phi_h^i - \psi_h^j\big)
+ \sum_{i=1}^M \phi_h^i \mu_h^i + \sum_{j=1}^N \psi_h^j \nu_h^j.
\end{split}\]
Minimization in $\pi_h\ge 0$ and maximization in $\phi_h$ and $\psi_h$ provide 
the condition 
\[
c(x_i,y_j) - \phi_h^i - \psi_h^j \ge 0,
\]
and the implication 
\[
\phi_h^i + \psi_h^j < c(x_i,y_j) \quad \implies \quad \pi_h^{ij} = 0,
\]
which determines the support of the discrete transport plan $\pi_h$. 

\subsection{Sparsity} 
The Knott--Smith theorem and generalizations thereof state 
that optimal transport plans are supported
on $c$-cyclically monotone sets, cf.~\cite{Vil08}. In particular, if $c$ is strictly 
convex and if the marginal $\mu$ has a density then optimal transport plans
are unique and supported on the graph of the $c$-subdifferential of a convex
function~$\Phi$. For the special case of a quadratic cost function 
it follows that $\Phi$ is a solution of the Monge--Amp\`ere equation 
for which regularity properties can be established, cf.~\cite{Vil03,PhiFig13}. Hence, in
this case it is rigorously established that the support is contained
in a lower-dimensional submanifold. Typically, such a quantitative behaviour 
can be expected but may be false under special circumstances. We refer
the reader to~\cite{CheFig17} for further details on partial regularity 
properties of transport maps. 

On the discrete level it is irrelevant to distinguish measures with or
without densities since the action of a discrete measure on a finite-dimensional
set $V_h$ of continuous functions can always be identified with an integration, 
i.e., we associate a well defined density $f_h\in V_h$ by requiring that
\[
\int_X v_h f_h \dv{x} = \dual{\mu_h}{v_h},
\]
for all $v_h\in V_h$. The properties
of optimal transport plans thus apply to the discrete transport problem
introduced above. Asymptotically, these properties remain valid provided
that we have $f_h \to f$ in $L^1(X)$ for a limiting density $f\in L^1(X)$. 


\section{Error analysis} \label{sec_erroranalysis}
We derive an error estimate for the approximation of the continuous 
problem $(\hP)$ by the discrete problem $(\hP_h)$ by appropriately 
interpolating measures. For this we follow~\cite{Roub97} and assume
that we are given a triangulation $\cT_h$ with maximal mesh-size~$h>0$
of a domain $U \subset \R^d$ which represents $X$ or $Y$ with nodes 
\[
\cN_h = \{z_1,z_2,\dots,z_L\}
\]
and associated nodal basis functions $(\vphi_z: z\in \cN_h)$. 
With the corresponding nodal interpolation operator
\[
\cI_h : C(U) \to \cS^1(\cT_h), \quad 
\cI_h v = \sum_{z\in \cN_h} v(z) \vphi_z,
\]
we define approximations $\cI_h^* \vrho$ of measures 
$\vrho \in \cM(U) \simeq C(U)^*$ via
\[
\dual{\cI_h^* \vrho}{u} = \dual{\vrho}{\cI_h u} 
= \sum_{z\in \cN_h} u(z) \dual{\vrho}{\vphi_z},
\]
i.e., we have the representation 
\[
\cI_h^* \vrho = \sum_{z\in \cN_h} \vrho_z \d_z
\]
with $\vrho_z = \dual{\vrho}{\vphi_z}$. Standard nodal interpolation
estimates imply that we have, cf.~\cite{BreSco08}, 
\[
\big|\dual{\vrho- \cI_h^* \vrho}{u}\big| = \big|\dual{\vrho}{u-\cI_h u}\big|
\le c_\cI h^{1+\a} \|u\|_{C^{1,\a}(U)} \|\vrho\|_{\cM(U)},
\]
for all $u\in C^{1,\a}(U)$. 
Analogously, we can approximate measures on the product space $X\times Y$
with triangulations $\cT_{X,h}$ and $\cT_{Y,h}$, nodes $\cN_{X,h}$ and
$\cN_{Y,h}$, and interpolation operators $\cI_{X,h}$ and $\cI_{Y,h}$, 
respectively, via
\[
\dual{\cI_{X\otimes Y,h}^* \pi}{r} = \dual{\pi}{\cI_{X\otimes Y,h} r}
= \sum_{(x,y)\in \cN_{X,h} \times \cN_{Y,h}} r(x,y) \dual{\pi}{\vphi_x \otimes \vphi_y},
\]
for all $r\in C(X\times Y)$. 
In the following error estimate we abbreviate the optimal values of
the minimization problems $(\hP)$ and $(\hP_h)$ by $\min_{\pi \ge 0} \hI[\pi]$
and $\min_{\pi_h \ge 0} \hI_h [\pi_h]$, respectively. 

\begin{proposition} \label{prop_energy_conv}
Assume that $\mu_h = \cI_{X,h}^* \mu$ and $\nu_h = \cI_{Y,h}^*\nu$. 
If $c\in C^{1,\a}(X\times Y)$ with $\a\in [0,1]$ we then have
\[
\big| \min_{\pi \ge 0} \hI[\pi] - \min_{\pi_h \ge 0} \hI_h[\pi_h]\big| 
\le c_\cI h^{1+\a} \|c\|_{C^{1,\a}(X\times Y)}.
\]
\end{proposition}

\begin{proof}
(i) Assume that $\min_{\pi\ge 0} \hI[\pi] \le \min_{\pi_h \ge 0} \hI_h[\pi_h]$. 
The interpolant $\tpi_h = \cI_{X\times Y,h}^* \pi$ of a solution $\pi$ for $(\hP)$ 
is admissible in $(\hP_h)$ since 
\[
\dual{\cI_{X\otimes Y,h}^*\pi}{v\otimes 1} = \dual{\pi}{\cI_{X,h} v \otimes 1} 
= \dual{\mu}{\cI_{X,h} v} = \dual{\cI_{X,h}^*\mu}{v} = \dual{\mu_h}{v},
\]
for every $v\in C(X)$, i.e., $P_X \tpi_h = \mu_h$. Analogously,
we find that $P_Y \tpi_h = \nu_h$. This implies that 
\[\begin{split}
\min_{\pi_h \ge 0} \hI_h[\pi_h] - \min_{\pi\ge 0} \hI[\pi] 
 &\le \hI_h[\tpi_h] - \hI[\pi] \\
 & = \dual{\tpi_h-\pi}{c} \le c_\cI h^{1+\a} \|c\|_{C^{1,\a}(X\times Y)},
\end{split}\]
where we used that $\|\pi\|_{\cM(X\times Y)} = 1$. \\
(ii) If conversely we have $\min_{\pi\ge 0} \hI[\pi] \ge \min_{\pi_h \ge 0} \hI_h[\pi_h]$
we let $\pi_h$ be a discrete solution and consider the measure
\[
\tpi = \pi_h + \dv{x} \otimes (\nu-\nu_h) + (\mu - \mu_h) \otimes \dv{y},
\]
which satisfies
\[
\dual{\tpi}{r} = \dual{\pi_h}{r} + \int_X \dual{\nu-\nu_h}{r(x,\cdot)} \dv{x}
+ \int_Y \dual{\mu-\mu_h}{r(\cdot,y)} \dv{y},
\]
for all $r\in C(X\times Y)$. We have that
\[
\dual{\tpi}{v\otimes 1} = \dual{\mu_h}{v} + \int_Y \dual{\mu-\mu_h}{v} \dv{y}
= \dual{\mu}{v},
\]
i.e., $P_X \tpi = \mu$. Analogously, we find that $P_Y \tpi = \nu$. 
Therefore, $\tpi$ is admissible in the minimization problem $(\hP)$ and hence
\[\begin{split}
\min_{\pi\ge 0}  \hI[\pi] - \min_{\pi_h \ge 0}& \hI_h[\pi_h] 
 \le \hI[\tpi] - \hI[\pi_h]  \\
& = \int_X \dual{\nu-\nu_h}{c(x,\cdot)}\dv{x} + \int_Y \dual{\mu-\mu_h}{c(\cdot,y)} \dv{y}\\
&\le c_\cI h^{1+\a} \big(\max_{x\in X} \|c(x,\cdot)\|_{C^{1,\a}(Y)} 
  + \max_{y\in Y} \|c(\cdot,y)\|_{C^{1,\a}(X)}\big) \\
&\le c_\cI h^{1+\a} \|c\|_{C^{1,\a}(X\times Y)},
\end{split}\]
where we used the property $\|\mu\|_{\cM(X)} = \|\nu\|_{\cM(Y)} = 1$. 
\end{proof}

The estimate can be improved if assumptions on the transport plan 
are made.

\begin{remark}
For the polynomial cost function $c_p(x,y) = (1/p)|x-y|^p$, $1\le p < \infty$,
we have $c_p\in C^{1,\a}(X\times Y)$ for $\a = \min\{1,p-1\}$, so that 
the derived convergence rate is subquadratic if $p<2$. If the 
transport plan is supported away from the diagonal $\{x=y\}$, along which 
the differentiability of $c_p$ is limited, then quadratic convergence applies.
\end{remark}

A similar error estimate is expected to hold if the measures
$\mu$ and $\nu$ are approximated via piecewise affine densities
$f_h$ and $g_h$ as this corresponds to a rescaling of coefficients
and the use of quadrature in the cost functional. 

\begin{remark}
Alternatively to the above discretization, transport plans can be
approximated via discrete measures $\pi_h$ which
have densities $p_h\in \cS^1(\cT_h^X)\otimes \cS^1(\cT_h^Y)$, i.e., 
\[
\dual{\pi_h}{r} = \iint_{X\times Y} r(x,y) p_h(x,y) \dv{(x,y)}
\]
with
\[
p_h(x,y) = \sum_{i=1}^M \sum_{j=1}^N p_h^{ij} \vphi_{x_i}(x) \vphi_{y_j}(y).
\]
We associate discrete densities $f_h\in \cS^1(\cT_h^X)$ and $g_h\in \cS^1(\cT_h^Y)$
with the marginals $\mu$ and $\nu$ via
\[
(f_h,v_h)_h = \dual{\mu}{v_h}, \quad (g_h,w_h)_h = \dual{\nu}{w_h},
\]
for all $v_h\in \cS^1(\cT_h^X)$ and $w_h\in \cS^1(\cT_h^Y)$ and with (discrete) inner
products $(\cdot,\cdot)_h$ on $C(X)$ and $C(Y)$, e.g., 
if $\mu$ and $\nu$ have densities $f$ and $g$ then
$f_h$ and $g_h$ may be defined as their $L^2$ projections.
If the inner products involve quadrature then we have
\[
(f_h,v_h)_h = \int_X \cI_{X,h} [f_hv_h] \dv{x} = \sum_{i=1}^M \b_i f_h(x_i) v_h(x_i),
\]
where $\b_i = \int_X \vphi_{x_i} \dv{x}$ and it follows that
\[
f_h(x_i) = \b_i^{-1} \dual{\mu}{\vphi_{x_i}}
\]
for $i=1,2,\dots,M$. Analogously, we have $g_h(y_j) = \g_j^{-1}\dual{\nu}{\vphi_{y_j}}$. 
The coefficients are thus scaled versions of the coefficients used above. 
Using quadrature in the cost functional leads to 
\[
I[\pi_h] = \iint_{X\times Y} c(x,y) p_h(x,y) \dv{(x,y)} 
\approx \sum_{i=1}^M \sum_{j=1}^N c(x_i,y_j) p_h^{ij} \b_i \g_j.
\]
Again, the coefficients here are scaled versions of the coefficients
$\pi_h^{ij}$ used above. 
\end{remark}

A reduced convergence rate applies for the approximation using 
piecewise constant finite element functions. 

\begin{remark}\label{rem:red_conv_p0}
Approximating measures by measures with densities that 
are elementwise constant, i.e., 
\[
\dual{\mu_h}{v} = \sum_{T\in \cT_h} \mu_h^T \int_T v \dv{x},
\]
we obtain a reduction of the convergence rate by one order. 
\end{remark}


\section{Active Set Strategy} \label{sec_strategy}
For a subset of atoms specified via an index set
\[
\cA \subset \{1,\dots,M\}\times \{1,\dots,N\}
\]
which is admissible in the sense that there exists $\tpi_h$ with
\[
\sum_{j=1,\dots,N, \, (i,j)\in \cA} \tpi_h^{ij} = \mu_h^i, \quad
\sum_{i=1,\dots,M, \, (i,j)\in \cA} \tpi_h^{ij} = \nu_h^j, 
\]
we restrict to discrete transport plans that are supported on $\cA$ and
hence consider the following reduced problem:
\[
(\hP_{h,\cA}) \quad \left\{ 
\begin{array}{l}
\text{Minimize } \hI_{h,\cA}[\pi_h] = \sum_{(i,j)\in \cA} c(x_i,y_j)  \pi_h^{ij} \\[1.5mm]
\text{subject to } \pi_h \ge 0, \ \sum_{j,(i,j)\in \cA} \pi_h^{ij} = \mu_h^i,
\ \sum_{i,(i,j)\in \cA} \pi_h^{ij} = \nu_h^j
\end{array} 
\right. 
\]

The following proposition provides a sufficient condition for
the definition of an active set that leads to an accurate reduction. 

\begin{proposition} \label{prop_as}
Assume that we are given approximations $\tphi_h$ and $\tpsi_h$ of 
exact discrete multipliers $\phi_h$ and $\psi_h$ with
\[
\|\tphi_h-\phi_h\|_{L^\infty(X)} + \|\tpsi_h - \psi_h\|_{L^\infty(Y)} \le \veps_{as}.
\]
If the set of active atoms $\cA$ on $X\times Y$ is defined via
\[
\cA = \big\{(i,j): \tphi_h^i + \tpsi_h^j \ge c(x_i,y_j) - 2 c_{as} \veps_{as} \big\}
\]
with $c_{as} \ge 1$
then the minimization problem $(\hP_{h,\cA})$ is an accurate reduction
of $(\hP_h)$ in the sense that their solution sets coincide. 
\end{proposition}

\begin{proof}
Let $\pi_h$ be a solution of the nonreduced problem~$(\hP_h)$ 
and let $\phi_h,\psi_h$ be corresponding Lagrange multipliers. If 
$\pi_h^{ij} \neq 0$ for the pair $(i,j)\in \{1,\dots,M\}\times \{1,\dots,N\}$
then we have $c(x_i,y_j) = \phi_h^i + \psi_h^j$ and hence
\[
\tphi_h^i + \tpsi_h^j  = \tphi_h^i - \phi_h^i + \tpsi_h^j - \psi_h^j + c(x_i,y_j)
\ge c(x_i,y_j) -2 c_{as} \veps_{\rm as}.
\]
This implies that $(i,j) \in \cA$ and $\pi_h$ is admissible in the
reduced formulation~$(\hP_{h,\cA})$. 
\end{proof}

Proposition \ref{prop_as} suggests a multilevel iteration realized
in the subsequent algorithm where the 
Lagrange multipliers of a coarse-grid solution are used as approximations 
for the multipliers on a finer grid which serve to guess the support
of the optimal transport plan. If the optimality conditions are not
satisfied up to a mesh-dependent tolerance then the variable activation 
tolerance is enlarged and the solution procedure
repeated. Because of the quasioptimal quadratic convergence behaviour
of the employed $P1$ finite element method, a quadratic tolerance is used. 

\begin{algorithm}[Multilevel active set strategy] \label{alg_as}
Choose triangulations $\cT_{X,h}$ and $\cT_{Y,h}$ of $X$ and $Y$
with maximal mesh-size $h>0$. Let $\theta_{act}>0$, $0< h_{min} < h$, and 
$c_{opt} > 0$. Choose functions $\tphi_h \in \cS^1(\cT_{X,h})$ and 
$\tpsi_h\in \cS^1(\cT_{Y,h})$. 
\begin{enumerate}
\item[(1)] Define the set of activated atoms via
\[
\cA = \big\{(i,j) : \tphi_h^i + \tpsi_h^j \ge c(x_i,y_j) - \theta_{act} h^2\big\}
\]
and enlarge $\cA$ to guarantee feasibility. 
\item[(2)] Solve the reduced problem $(\hP_{h,\cA})$ and extract multipliers $\phi_h$ and $\psi_h$.
\item[(3)] Check optimality conditions up to tolerance $c_{opt} h^2$ on the 
full set of atoms, i.e., whether 
\[
\phi_h^i + \psi_h^j \le c(x_i,y_j) + c_{opt} h^2,
\]
is satisfied for all $(x_i,y_j) \in \cN_{X,h}\times \cN_{Y,h}$.
\item[(4)] If optimality holds and $h>h_{min}$ then refine  
triangulations $\cT_{X,h}$ and $\cT_{Y,h}$, prolongate functions $\phi_h$ 
and $\psi_h$ to the new triangulations with new mesh-size $h\leftarrow h/2$\
to update $\tphi_h$ and $\tpsi_h$, 
set $\theta_{act} \leftarrow \theta_{act}/2$, and continue with~(1).
\item[(5)] If optimality fails then set $\theta_{act} \leftarrow 2 \theta_{act}$ and continue
with~(1). 
\item[(6)] Stop if optimality holds and $h\le h_{min}$. 
\end{enumerate}
\end{algorithm}

Various modifications of Algorithm~\ref{alg_as} are possible that may
lead to improvements of its practical performance. 

\begin{remarks} 
(i) The activation parameter $\theta_{act}$ is adapted during the procedure, 
i.e., the increased constant is used in a new iteration on one level. To avoid 
activating too many atoms initially, $\theta_{act}$ is decreased whenever a 
new level is reached. \\
(ii) The quadratic tolerance in the verification of the optimality conditions
turned out to be sufficient to obtain a quadratic convergence of optimal 
costs and of the Lagrange multipliers in our experiments. \\ 
(iii) The initial parameter $\theta_{act}$ can be optimized on the coarsest 
mesh by repeatedly reducing it until optimality fails. 
\end{remarks}


\section{Numerical experiments} \label{sec_exp}
In this section we illustrate our theoretical investigations via several 
experiments. We implemented Algorithm~\ref{alg_as} in {\sc Matlab} and used 
the optimization package {\sc Gurobi}, cf.~\cite{gurobi}, to solve the linear programs. 
The experiments were run on a 2012 MacBook Air (1.7 GHz Intel Core i5 with 4 GB RAM)
with {\sc Matlab} version R2015b. Integrals were evaluated using a three-point
trapezoidal rule on triangles. 
The employed triangulations result from uniform refinements
of an initial coarse triangulation and are represented via their refinement
level $k \in \N$ so that the maximal mesh-size satisfies
$h \sim 2^{-k}$. The number of nodes in the triangulations 
of the spaces $X$ and $Y$ are referred to by $M$ and $N$, respectively. 

\subsection{Problem specifications}
We consider four different transport problems specified via the sets $X$ and
$Y$ and the marginals $\mu$ and $\nu$ together with different polynomial cost 
functions 
\[
c_p(x,y) = \frac1p |x-y|^p,
\]
where $p \in \{ 3/2,2,3\}$. These choices are prototypical 
for subquadratic, quadratic, and superquadratic costs leading 
to singular, linear, and degenerate cost gradients, respectively.
In the special case of a quadratic cost solutions for the optimal transport
problem can be constructed using the Monge--Amp\`ere equation
\[
\det D^2 \Phi = \frac{f}{g\circ \nabla \Phi}
\]
and the relations for the transport map and the multipliers 
\[
T = \nabla \Phi, \quad \phi(x) = \frac{|x|^2}{2} - \Phi (x), \quad 
\psi(y) = \frac{|y|^2}{2} - \Phi^*(y), 
\]
with the convex conjugate $\Phi^*(y) = \sup_x x\cdot y - \Phi(x)$ 
of $\Phi$, cf.~\cite{Vil08} for details. 
Moreover, we then have the optimal cost
\[
I[T] = I[\nabla \Phi] = \int_X c_2(x,\nabla \Phi(x)) \dv{\mu(x)}.
\]
The first example is one-dimensional and allows for a simple visualization
of the transport map. 

\begin{example}[One-dimensional transport] \label{ex_1d}
Let $X=Y=[0,1]$ and $\mu$ and $\nu$ be defined via the densities
\[
f(x) = \frac23 (x + 1), \quad g(y) = 1,
\] 
respectively. For $p=2$ the optimal transport plan is given by 
the transport map $T=\nabla \Phi$ with the potential 
\[
\Phi: [0,1] \rightarrow \R, \quad x \mapsto \frac{1}{9} x^3 + \frac{1}{3} x^2,
\]
and the Lagrange multiplier $\phi$ satisfies
\[
\phi(x) = \frac{1}{6} x^2 - \frac{1}{9} x^3.
\]
The optimal cost for $p=2$ is given by $I[T] = 1/540$.
\end{example}

Our second example concerns the transport between two rectangles with
a differentiable transport map. 

\begin{example}[Smooth transport between rectangles] \label{ex_mae1}
Defining $X=[0,1]^2$ and $\Phi(x_1,x_2) = x_1^2 + x_2^3$ we 
set $Y = \nabla  \Phi(X) = [0,2] \times [0,3]$ and $g=1$. The 
Monge--Amp\`ere equation determines 
\[
f(x_1,x_2) = 12x_2,
\]
so that the optimal cost value for $p=2$
is given by $I[\nabla \Phi] = 43/10$.
\end{example}

In order to compare our algorithm to the results from~\cite{ObeRua15}
we incorporate Example~4.1 from that article. 

\begin{example}[Setting from~\cite{ObeRua15}] \label{ex_41}
On $X = Y = [-1/2,1/2]^2$, let $\mu$ and $\nu$ be defined by the 
densities 
\begin{align*}
f(x_1,x_2) = \ &1 + 4 ( q''(x_1)q(x_2) + q(x_1)q''(x_2) ) \\ 
&+ 16 (q(x_1)q(x_2)q''(x_1)q''(x_2)-q'(x_1)^2 q'(x_2)^2)
\end{align*}
and $g=1$, where
\begin{align*}
q(z) = \left( -\frac{1}{8\pi}z^2 + \frac{1}{256\pi^3} 
+ \frac{1}{32\pi} \right) \cos(8\pi z) + \frac{1}{32\pi^2}z \sin(8\pi z).
\end{align*}
For $p=2$ we obtain an exact solution via the Monge--Amp\`ere equation. 
\end{example}

The final example describes the splitting of a square into two 
rectangles. 

\begin{example}[Discontinuous transport] \label{ex_nonsmooth}
Let $X = [-1/2,1/2]^2$ and $Y = \big([-3/2,-1]  \cup [1,3/2]\big) \times [-1/2,1/2]$ 
be equipped with the constant densities $f = 1$ and $g = 1$. For any strictly convex 
cost function, optimal transport maps $T$ isometrically map the left half of the square 
to the rectangle on the left side and the other half to the one on the right, i.e.,
up to identification of Lebesgue functions, 
\[
T(x_1,x_2) = 
\begin{cases}
(x_1+1,x_2) & \text{if } x_1 > 0, \\ 
(x_1-1,x_2) & \text{if } x_1 < 0. 
\end{cases}
\]
For $p=2$ we have $T = \nabla \Phi$ with 
\[
\Phi(x_1,x_2) = \frac{x_1^2 + x_2^2}{2} + |x_1|,
\]
with corresponding Lagrange multiplier $\phi(x_1,x_2) = -|x_1|$.
\end{example}

Figure~\ref{fig:examples} shows characteristic features of the
four examples. In particular, in the upper left plot of Figure~\ref{fig:examples}
the transport plan is the graph of a monotone function and we illustrated
an activated set of atoms of a discretization that approximates the 
graph. 

\begin{figure}
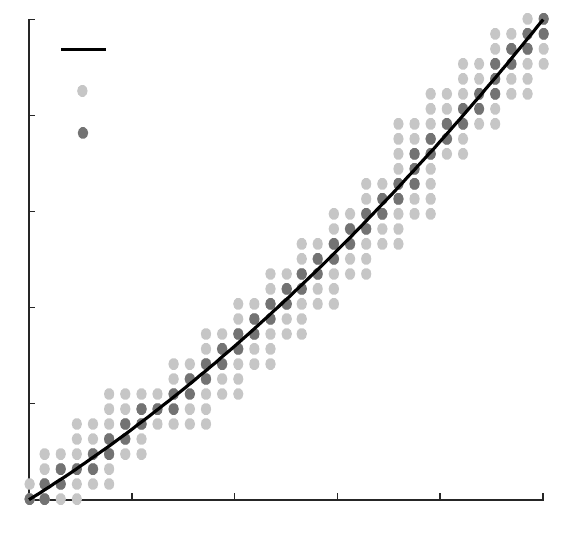\hspace*{16mm}
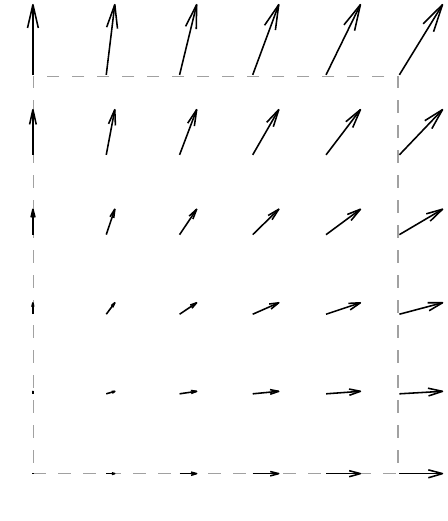 \\
\includegraphics[width=6.3cm,height=5.5cm]{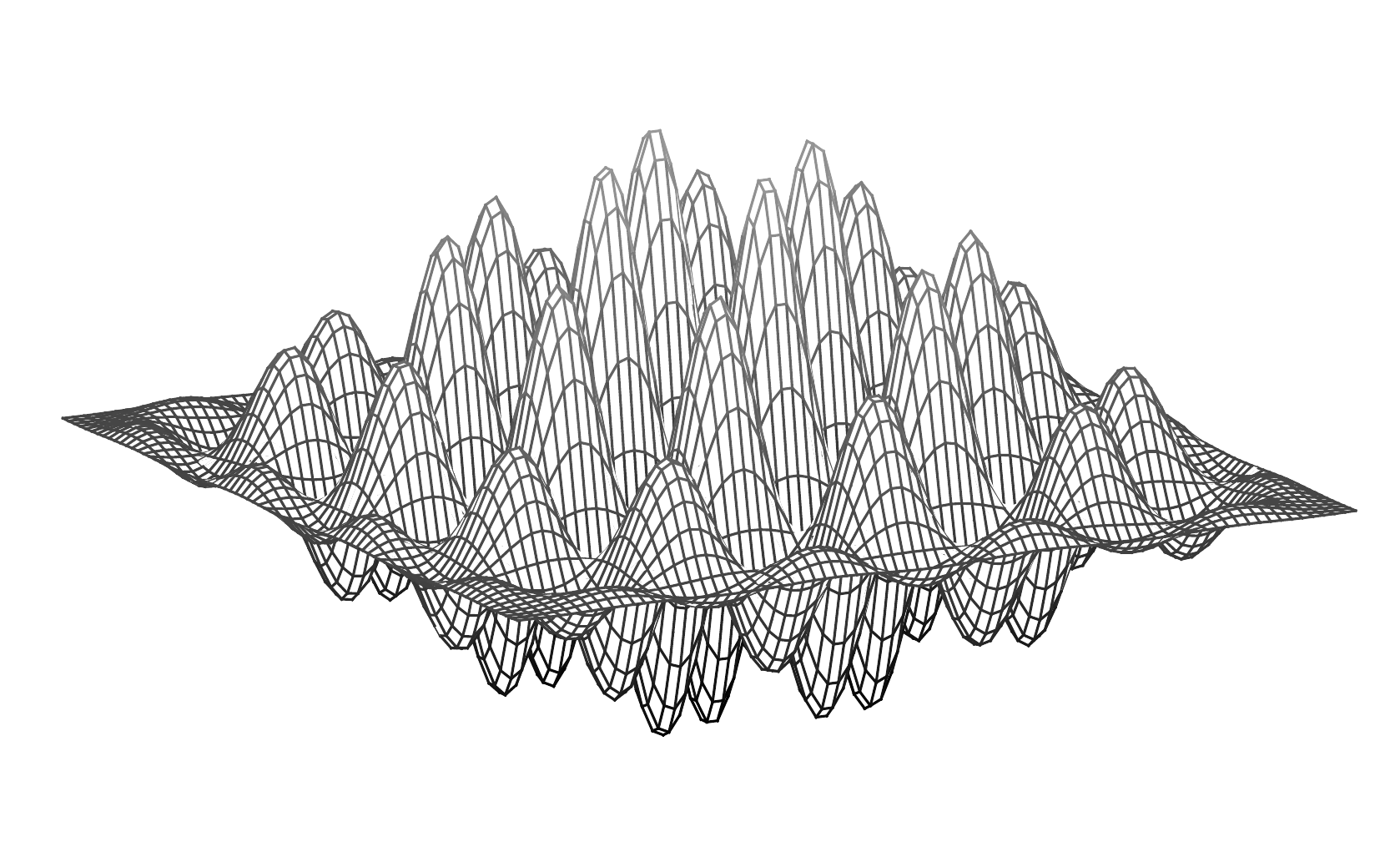} \hspace*{2mm} 
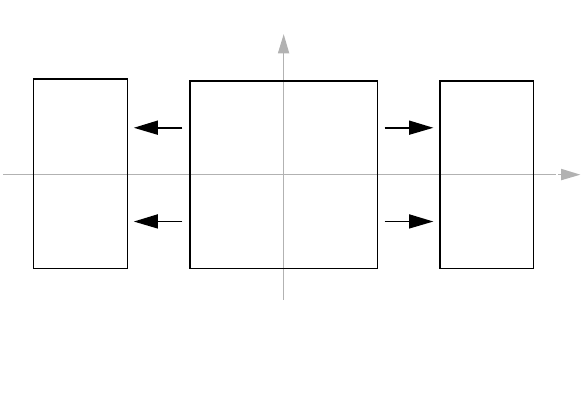  
\caption{\label{fig:examples} Characteristic features of
the transport problems defined in Examples~\ref{ex_1d}-\ref{ex_nonsmooth}
(from left to right and top to bottom):
(i)~optimal transport plan given by a graph together with activated atoms
and discrete support for $k=5$ in Example~\ref{ex_1d}, 
(ii)~optimal transport map $T=\nabla \Phi$ in Example~\ref{ex_mae1} interpreted as a vector field,
(iii)~oscillating density $f$ in Example~\ref{ex_41}, 
(iv)~piecewise affine optimal transport plan $T$ in Example~\ref{ex_nonsmooth}.}
\end{figure}

\subsection{Complexity considerations}
A crucial quantity to determine the efficiency of our devised
method is the growth of the cardinalities of the activated sets. 
In Table~\ref{tab_3size} we display for
Examples~\ref{ex_1d}-\ref{ex_nonsmooth} the corresponding numbers
on different triangulations and for different cost functions. 
We observe that in all experiments the size of the activated sets
grows essentially linearly in strong contrast to the quadratic growth of the 
theoretical number of unknowns of the corresponding discrete
transport problem. A slight deviation of this behaviour occurs in Example~\ref{ex_mae1}
for $p=3$ where the increase of the active set size is larger than
the expected factor~4. We note that we observed a reduction of the
active set sizes by factors of approximately~$2^{-d}$ compared to 
the sizes obtained with the algorithm from~\cite{ObeRua15} for generic
choices of parameters. Because of the very few required redefinitions
of the active set, particularly for $p\ge 2$, we conclude that the
optimality conditions provide a precise prediction of the supports
even if only approximations of the multipliers are available, i.e.,
this property appears to be very robust with respect to perturbations
of the multipliers. 

\begin{table}[htb]
	\begin{align*}
	\begin{array}{l||r|r|r|r} \hline\hline
 	{\rm Ex.}~\ref{ex_1d} & k=7 & k=8 & k=9 & k=10 \\
 	\hline
 	M+N & 258 & 514 & 1.026 & 2.050 \\
 	MN  & 16.641 & 66.049 & 263.169 & 1.050.625 \\
 	\hline
 	p = 3/2 & 763 \ (0) & 1.531 \ (0) & 3.067 \ (0) & 6.139 \ (0) \\
 	p = 2 & 763 \ (0) & 1.531 \ (0) & 3.067 \ (0) & 6.139 \ (0) \\
 	p = 3 & 763 \ (0) & 1.539 \ (0) & 3.114 \ (0) & 6.442 \ (0) \\[1mm] \hline\hline  
	{\rm Ex.}~\ref{ex_mae1} & k=3 & k=4 & k=5 & k=6 \\
	\hline
	M+N & 506 & 1.906 & 7.394 & 29.122 \\
	MN & 34.425 & 467.313 & 6.866.145 & 105.189.825 \\
	\hline
	p = 3/2 & 6.268 \ (8) & 27.846 \ (1) & 179.594 \ (2) & 745.713 \ (1) \\
	p = 2 & 3.929 \ (0) & 15.729 \ (0) & 63.115 \ (0) & 252.951 \ (0) \\
	p = 3 & 8.085 \ (2) & 56.703 \ (2) & 255.965 \ (1) & 1.847.207 \ (2) \\[1mm] \hline\hline  
	{\rm Ex.}~\ref{ex_41} & k=3 & k=4 & k=5 & k=6 \\
	\hline
	M+N & 162 & 578 & 2.178 & 8.450 \\
	MN & 6.561 & 83.521 & 1.185.921 & 17.850.625 \\
	\hline
	p = 3/2 &1.389 \ (0) & 20.787 \ (7) & 58.575 \ (1) & 183.465 \ (1) \\
	p = 2 & 1.589 \ (0) & 5.755 \ (0) & 24.018 \ (0) & 103.100 \ (0) \\
	p = 3 & 1.495 \ (0) & 6.319 \ (0) & 26.205 \ (0) & 106.857 \ (0) \\[1mm] \hline\hline 
	{\rm Ex.}~\ref{ex_nonsmooth}  &  k=3 & k=4 & k=5 & k=6 \\
	\hline
	M+N & 171 & 595 & 2.211 & 8.515 \\
	MN & 7.290 & 88.434 & 1.21.858 & 18.125.250 \\
	\hline
	p = 3/2 & 1.346 \ (0) & 6.384 \ (0) & 24.135 \ (0) & 95.240 \ (0) \\
	p = 2 & 1.654 \ (0) & 6.921 \ (0) & 29.106 \ (0) & 120.153 \ (0) \\
	p = 3 & 1.274 \ (0) & 5.602 \ (0) & 21.353 \ (0) & 85.463 \ (0) \\[1mm] \hline\hline 
        \end{array}
	\end{align*}
\caption{\label{tab_3size} Total number of nodes $M+N$, number of unknowns 
in the full optimization problem $MN$, and cardinalities of activated sets at 
optimality with number of tolerance increases in brackets in Examples~\ref{ex_1d}-\ref{ex_nonsmooth}
on triangulations with refinement level $k$ and different 
cost functions $c_p(x,y)$.} 
\end{table}	

In Table~\ref{tab_5size} we display the total CPU time needed to
solve the optimization problem on the $k$-th level. This includes
the repeated activation of atoms, the repeated solution of the
reduced linear programs, and the verification of the optimality
conditions. We observe a superlinear growth of the numbers. These
are dominated by the times needed to solve the linear programs 
whereas the (non-parallelized) verification of the optimality conditions was negligible
in all tested situations. 

\begin{table}[htb]
	\begin{align*}
	\begin{array}{l||r|r|r|r} \hline\hline
	{\rm Ex.}~\ref{ex_1d} & k=7 & k=8 & k=9 & k=10 \\
 	\hline
	p = 3/2 & 0.0575 & 0.1357 & 0.4281 & 1.1840 \\
	p = 2 & 0.0603 & 0.1288 & 0.3294 & 1.0257 \\
	p = 3 & 0.0593 & 0.1345 & 0.3943 & 1.3115\\[1mm] \hline\hline  
	{\rm Ex.}~\ref{ex_mae1} & k=3 & k=4 & k=5 & k=6 \\
	\hline
	p = 3/2 & 0.2063 & 0.8805 & 6.9296 & 49.5964 \\
	p = 2 & 0.2187 & 0.5262 & 2.4734 & 21.1738 \\
	p = 3 & 0.2584 & 1.6697 & 7.3599 & 97.7239 \\[1mm] \hline\hline  
	{\rm Ex.}~\ref{ex_41} & k=3 & k=4 & k=5 & k=6 \\
	\hline
	p = 3/2 & 0.1090 & 0.7655 & 1.5510 & 9.5078 \\
	p = 2 & 0.1106 & 0.1718 & 0.7735 & 4.6622 \\
	p = 3 & 0.1410 & 0.1886 & 0.8557 & 5.6919 \\[1mm] \hline\hline 
	{\rm Ex.}~\ref{ex_nonsmooth}  &  k=3 & k=4 & k=5 & k=6 \\
	\hline
	p = 3/2 & 0.1192 & 0.1777 & 0.8014 & 4.9880 \\
	p = 2 & 0.1054 & 0.1766 & 0.6771 & 4.0459 \\
	p = 3 & 0.1214 & 0.1700 & 0.7194 & 4.8851 \\ \hline\hline
	\end{array}
	\end{align*}
	\caption{\label{tab_5size} Total CPU time in seconds on $k$-th level in 
        Examples~\ref{ex_1d}-\ref{ex_nonsmooth} with different polynomial 
        cost functions $c_p(x,y)$.}
\end{table}

\subsection{Experimental convergence rates}
In Figures~\ref{fig:eoc_cost} and~\ref{fig:eoc_multi}
we show for Examples~\ref{ex_1d} 
and~\ref{ex_mae1} the error in the approximation of the optimal cost, i.e.,
the quantities 
\[
\d_h = \big|\min_{\pi \ge 0} \hI[\pi] - \min_{\pi_h\ge 0}\hI_h[\pi_h]\big| 
\]
and the error in the approximation of the Lagrange multiplier $\phi$, i.e.,
the quantities
\[
\veps_h = \|\cI_{X,h} \phi - \phi_h\|_{L^\infty(X)}. 
\]
If the exact optimal cost or the multiplier was not known, i.e., 
if $p\neq 2$, we used
an extrapolated reference value or considered the difference 
$\cI_{X,h} \phi_{h/2}-\phi_h$ to define $\d_h$ and $\veps_h$,
respectively. 
We tested different polynomial costs and considered sequences of 
uniformly refined triangulations. Beause of the relation
\[
h \sim (M+N)^{-1/d},
\]
a quadratic convergence rate~$\cO(h^2)$ corresponds to a slope
$-2/d$ with respect to the total number of nodes $M+N$. 
Figure~\ref{fig:eoc_cost} confirms the estimate from 
Proposition~\ref{prop_energy_conv} and additionally shows that the 
quadratic convergence rate is optimal. The experimental results
also reveal that the employed quadratic tolerance in the
verification of the optimality conditions in Algorithm~\ref{alg_as}
is sufficient to preserve the convergence rate of the 
linear program using the full set of atoms. 
Figure~\ref{fig:eoc_multi} indicates that quadratic convergence
in $L^\infty(X)$ also holds for the approximation of the Lagrange multiplier
$\phi$ provided this quantity is sufficiently regular. In particular,
we observe here a slower convergence behaviour for $p=3/2$. 

\begin{figure}[htb]
\includegraphics[width=.7\linewidth]{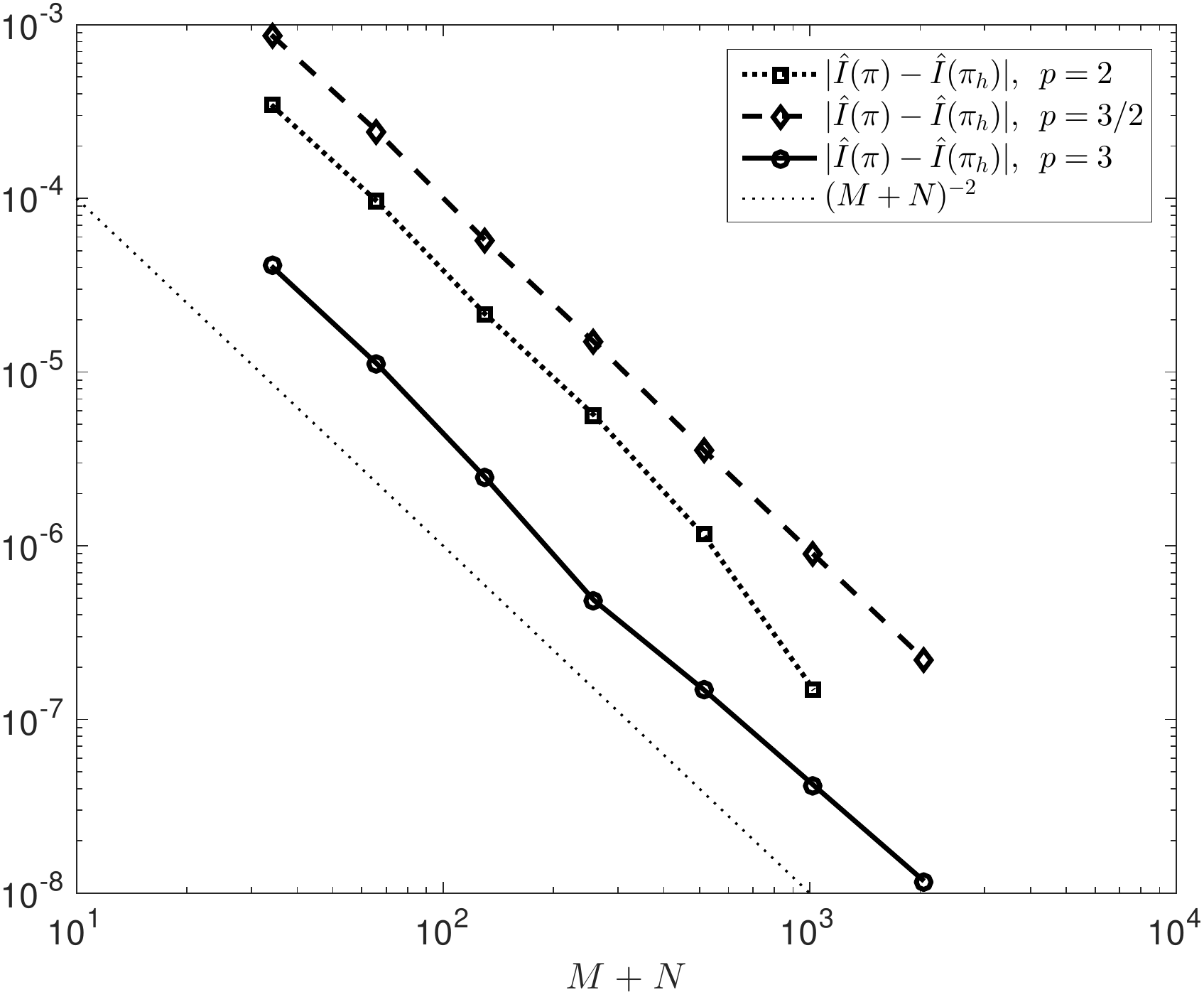}
\includegraphics[width=.7\linewidth]{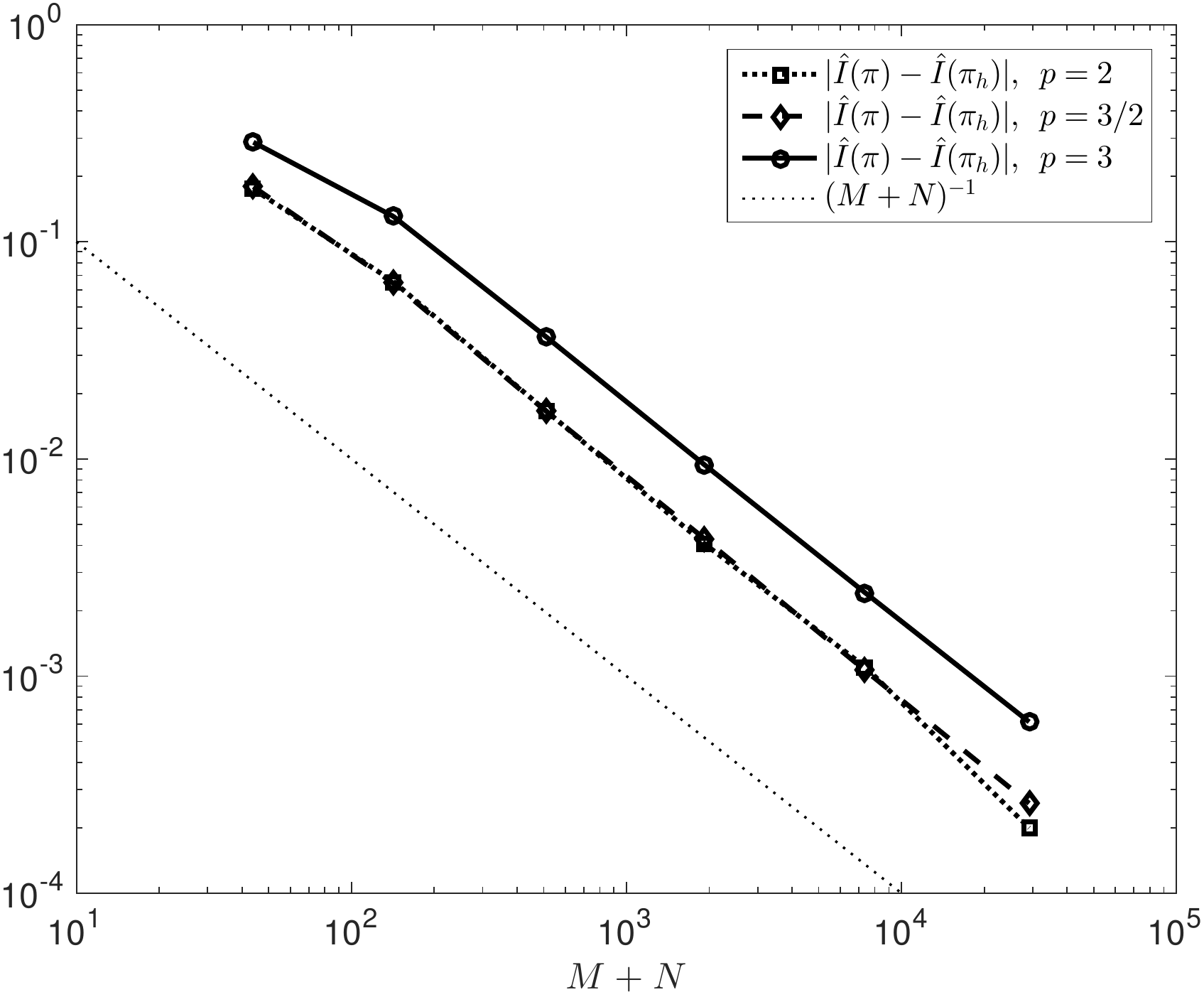}
\caption{\label{fig:eoc_cost} Experimental convergence of
optimal costs in Examples~\ref{ex_1d} (left)
and~\ref{ex_mae1} (right) for different cost functions
on sequences of uniformly refined triangulations.}
\end{figure}

\begin{figure}[htb]
\includegraphics[width=.7\linewidth]{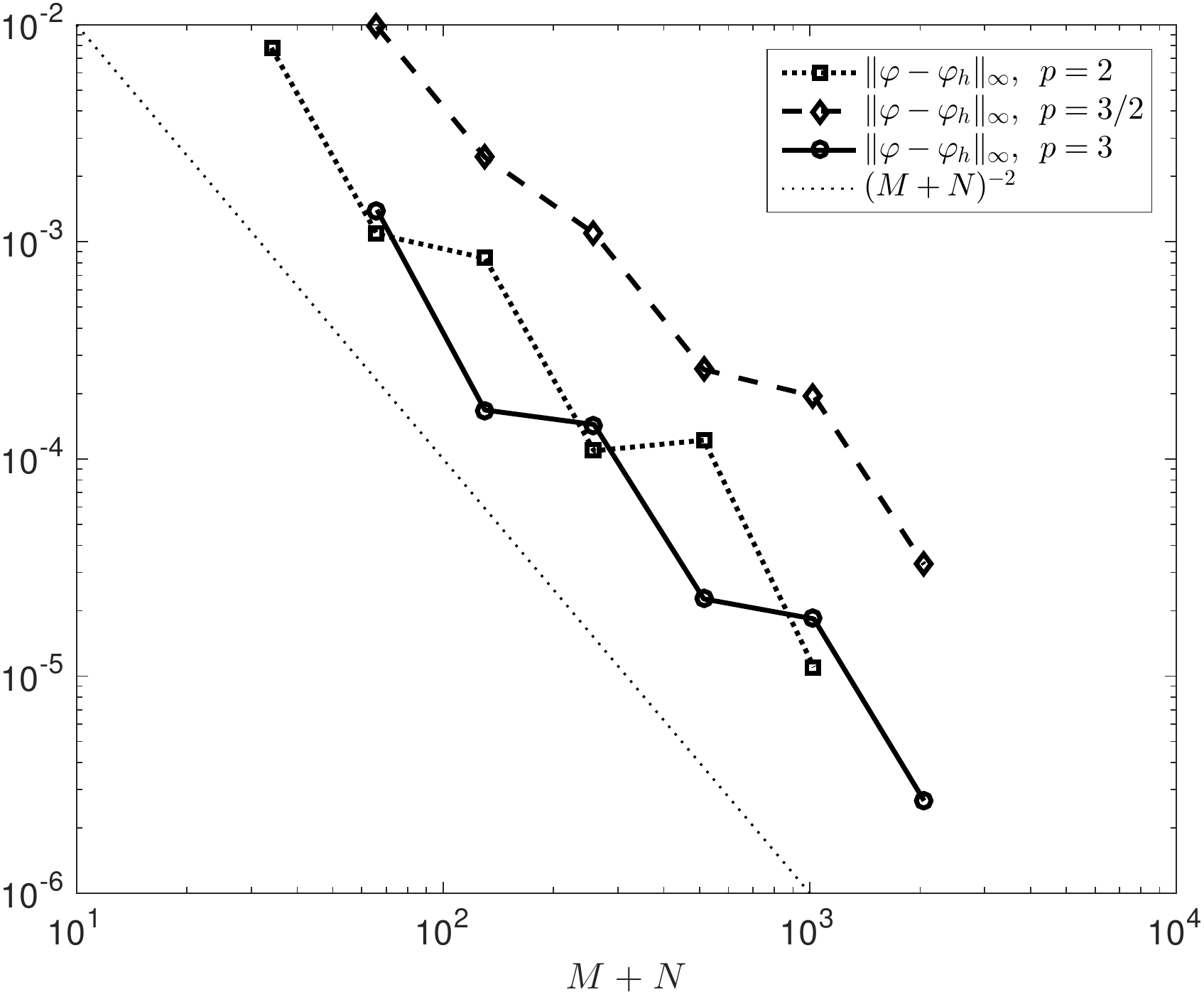} 
\includegraphics[width=.7\linewidth]{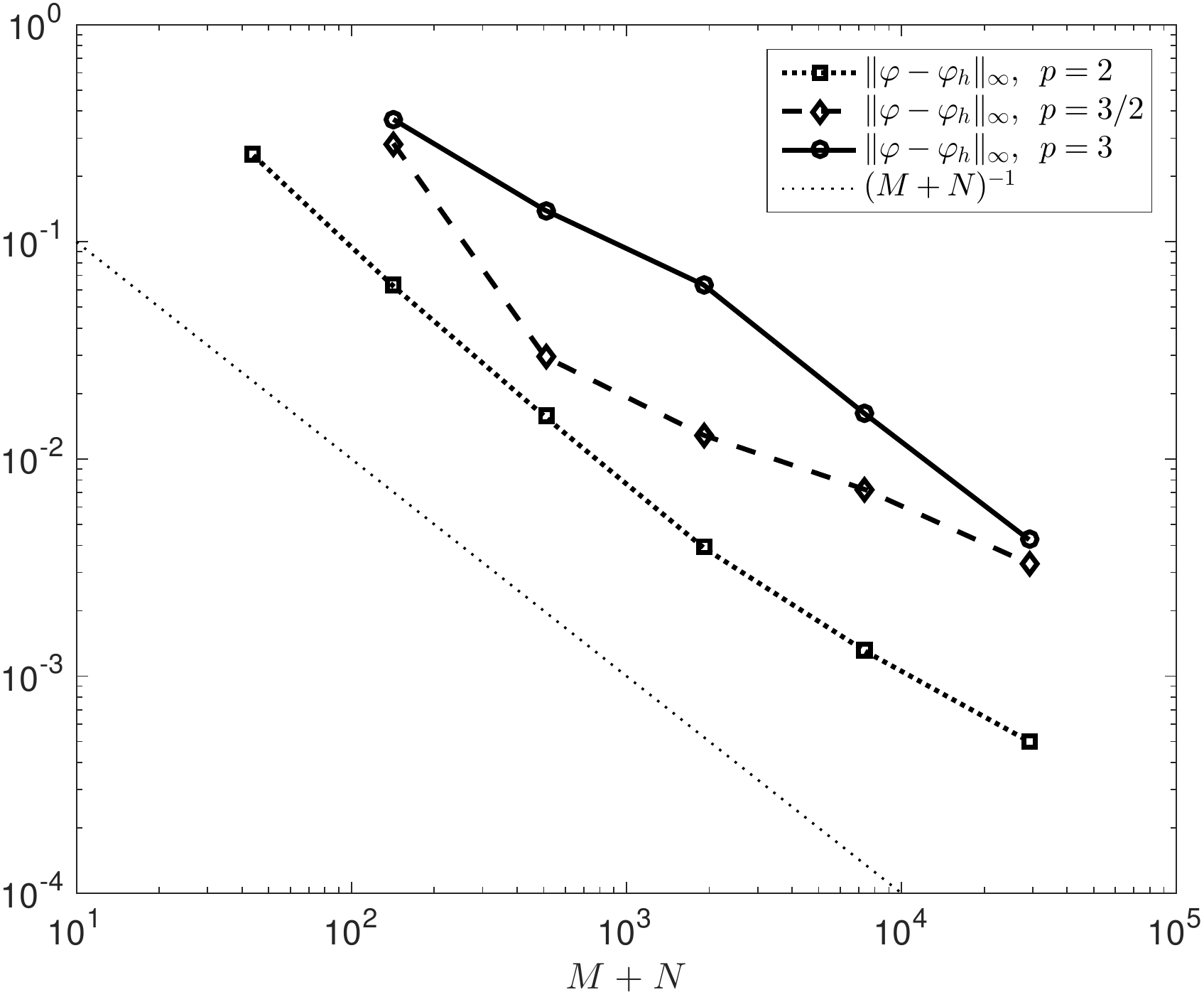}
\caption{\label{fig:eoc_multi} Experimental convergence 
rates of the discrete multiplier $\phi_h$ in Examples~\ref{ex_1d} (left)
and~\ref{ex_mae1} (right) for different cost functions
on sequences of uniformly refined triangulations.}
\end{figure}	

In~\cite{ObeRua15} an approximately linear convergence rate in $L^\infty$ of
the mulitpliers has been 
reported for Example~\ref{ex_41} which is consistent with the 
piecewise constant approximation of densities of measures used
in that article, cf. Remark~\ref{rem:red_conv_p0}.
In particular, discrete duality yields that
the Lagrange multipliers occurring in the discretized optimal
transport problems are discretized in the same spaces. 
For our discretization using continuous, piecewise
affine approximations we obtain a nearly quadratic 
experimental convergence rate in this 
example as well, as can be seen in Table~\ref{tab:comp_OR} in which
we also display the errors from~\cite{ObeRua15}. 

\begin{table}[htb]
\begin{tabular}{c||c|c|c|c|c}
$\veps_h$ & $h \sim 2^{-5}$ & $h\sim 2^{-6}$ & $h\sim 2^{-7}$ & $h\sim 2^{-8}$ & $h\sim 2^{-9}$ \\\hline
$P1$ (Alg.~\ref{alg_as})& 0.00781 & 0.00238 & 0.00086  & -- & -- \\
$P0$ (\cite{ObeRua15})  & 0.00721 & 0.00892 & 0.00689 & 0.00241 & 0.00148 
\end{tabular} 
\vspace*{2mm}
\caption{\label{tab:comp_OR} Experimental errors 
$\veps_h = \|\cI_h \phi - \phi_h\|_{L^\infty(\O)}$ for 
discretizations using $P1$ and $P0$ approximations 
of densities in Example~\ref{ex_41} with $p=2$.}
\end{table}

\medskip
\noindent
{\em Acknowledgments.} SB acknowledges support by
the DFG via the priority program {\em Non-smooth and 
Complementarity-based Distributed Parameter Systems: 
Simulation and Hierarchical Optimization} (SPP 1962).


\bibliographystyle{amsalpha}
\bibliography{bib_project}


\end{document}